\documentclass{article}
\usepackage{graphicx}
\usepackage{amsthm, amsmath, amssymb, tikz, bm}
\usepackage{mathtools}
\usepackage{xifthen}
\usepackage[normalem]{ulem}

\usepackage[margin=1.2in]{geometry} 

\usepackage[shortlabels]{enumitem}
\usepackage{todonotes}
\usepackage[colorlinks=true,
linkcolor=blue,citecolor=blue,
urlcolor=blue]{hyperref}

\allowdisplaybreaks

\usetikzlibrary{calc,shapes, backgrounds}

\usepackage{amssymb, amsmath, amsthm, graphicx,mathrsfs}
\usepackage{xcolor}

\usepackage{graphicx} 

\newtheorem{thm}{Theorem}[section]
\newtheorem{cor}[thm]{Corollary}

\newtheorem{lem}[thm]{Lemma}

\newtheorem{quest}[thm]{Question}
\newtheorem{definition}[thm]{Definition}

\newtheorem{claim}[thm]{Claim}

\title{Ore's Theorem for rainbow Hamiltonian-connected graphs}
\date{\today}
\author{
{{Yupei Li}}\thanks{
\footnotesize {University of South Carolina, Columbia, SC 29208, USA. Email: {\tt yupei@email.sc.edu}.}}
\and{{Ruth Luo}}\thanks{University of South Carolina, Columbia, SC 29208, USA. Email: {\tt ruthluo@sc.edu}. Research of this author
is supported in part by NSF grant DMS-2452134.
}}
\begin{document}

\maketitle

\begin{abstract}Let $G = (G_1, G_2, \ldots, G_m)$ be a collection of $m$ graphs on a common vertex set $V$. For a graph $H$ with vertices in $V$, we say that $G$ contains a rainbow $H$ if there is an injection $c: E(H) \to [m]$ such that for every edge $e \in E(H)$, we have $e \in E(G_{c(e)})$.

In this paper, we show that if $G = (G_1, \ldots, G_n)$ is a collection of graphs on $n$ vertices such that for every $i \in [n]$, $d_{G_i}(u) + d_{G_i}(v) \geq n$ whenever $uv \notin E(G_i)$, then either $G$ contains rainbow Hamiltonian paths between every pair of vertices, or $G$ contains a rainbow Hamiltonian cycle. Moreover, we prove a stronger version in which we may also embed prescribed rainbow linear forests into the Hamiltonian paths. 
\end{abstract}

\section{Introduction}

Suppose $G = (G_1, \ldots, G_m)$ is a collection of $m$ (not necessarily distinct) graphs with the same vertex set $V=V(G)$. For any graph $H$ with $V(H) \subseteq V(G)$, we say $G$ contains a {\bf rainbow $H$} if there exists an injection $c: E(H) \to [m]$ such that for every edge $e \in E(H)$, $e$ also belongs to $E(G_{c(e)})$. In other words, we can view the $m$ graphs $G_1, \ldots, G_m$ as corresponding to distinct colors $1, \ldots, m$, respectively. For each edge $e \in E(H)$, we have the option of assigning to it any color $i$ in which $e \in E(G_i)$. Then $G$ contains a rainbow $H$ if there exists a coloring of $E(H)$ in which each color is used at most once. 

In the case that $m = |E(H)|$, i.e., $c$ is a bijection, we say that $H$ is a {\bf $G$-transversal} if $G$ contains a rainbow $H$. This notion of $G$-transversals and rainbow subgraphs was first formalized by Joos and Kim in~\cite{JK}.

The famous Dirac's Theorem states that any graph $G$ on $n \geq 3$ vertices with minimum degree $\delta(G) \geq n/2$ contains a Hamiltonian cycle, and moreover this bound is best possible (for instance consider the complete bipartite graph $K_{\lfloor (n-1)/2 \rfloor, \lceil (n+1)/2\rceil}$). We refer to results in which some lower bound on the minimum degree $\delta(G)$ implies the existence of a certain substructure as {\em Dirac-type} results.

Joos and Kim~\cite{JK} proved the following rainbow version of Dirac's Theorem.

\begin{thm}[Joos, Kim~\cite{JK}]\label{JKthm} Suppose $G = (G_1, \ldots, G_n)$ is a collection of $n$ graphs with the same vertex set $V$ where $|V|=n\geq 3$ and $\delta(G_i) \geq n/2$ for all $i\in[n]$. Then $G$ contains a rainbow Hamiltonian cycle.     
\end{thm}

Again, the condition $\delta(G_i) \geq n/2$ is best possible because one could take $G$ to have $n$ identical copies of an extremal graph for Dirac's Theorem that is not Hamiltonian.

Joos and Kim also proved a sharp rainbow Dirac-type Theorem for perfect matching transversals, and Cheng, Sun, Wang, Wei ~\cite{CSWW} for Hamiltonian path transversals. Later, Sun, Wang, Wei ~\cite{SWW} proved Dirac-type theorems for tranversal Hamiltonian-connectedness and panconnectedness. We say a graph $G$ is {\bf Hamiltonian-connected} if between every pair of vertices $u,v \in V(G)$ there exists a $u,v$-Hamiltonian path, and is {\em panconnected} if between every pair of vertices $u,v \in V(G)$ there exists $u,v$-paths on $k$ vertices for all $k$ between their distance $d_G(u,v)$ and $|V(G)|$.  Bradshaw~\cite{B} proved Dirac-type Theorems for rainbow bipancyclicity in bipartite graphs, extending a Theorem by Schmeichel and Mitchem~\cite{SM}.
See~\cite{SWW2} for a survey on many more related rainbow subgraph and transversal problems.

A well known strengthening of Dirac's Theorem is due to Ore. We define the function 
\[\sigma_2(G) = \min_{uv \notin E(G)} d_G(u) + d_G(v).\]
\begin{thm}[Ore~\cite{Ore}]Let $G$ be a graph on $n \geq 3$ vertices such that $\sigma_2(G)\geq n$. Then $G$ contains a Hamiltonian cycle. 
\end{thm}

Towards an Ore-type Theorem for transversals, Li, Li, Li proved the following result for rainbow Hamiltonian paths.

\begin{thm}[Li, Li, Li~\cite{LLL}]\label{Li} Suppose $G = (G_1, \ldots, G_n)$ is a collection of $n$ graphs with the same vertex set $V$ where $|V|=n$ and $\sigma_2(G_i) \geq n-2$ for all $i \in [n]$. Then one of the following statements holds:
\begin{enumerate}
    \item[$(A1)$] $G$ has a rainbow Hamiltonian path;
    \item[$(A2)$] $G$ consists of $n$ identical copies of $K_\ell \cup K_{n-\ell}$ where $\ell \in [n-1]$;
    \item[$(A3)$] $n$ is even and there is a partition $V = X \cup Y$ with $|X| = n/2-1, |Y| = n/2+1,$ and for every $i \in [n]$, $G_i[Y]$ is an independent set.
\end{enumerate}
\end{thm}

In~\cite{B} and~\cite{LLL}, Bradshaw and Li, Li, Li asked if an Ore's Theorem for rainbow Hamiltonian cycles could be proved. This question is formalized below.

\begin{quest}\label{rainbowham}Suppose $G = (G_1, \ldots, G_n)$ is a collection of $n$ graphs with the same vertex set $V$ where $|V|=n$ and $\sigma_2(G_i) \geq n$ for all $i \in [n]$. Is it true that $G$ has a rainbow Hamiltonian cycle?
\end{quest}
In this paper, we prove a weaker version for Hamiltonian-connectedness.

\subsection{New results}
We say that $G = (G_1, \ldots, G_m)$ is {\bf rainbow Hamiltonian-connected} if for every pair $u,v \in V(G)$ there exists a rainbow Hamiltonian path from $u$ to $v$. Our main result is the following extension of Theorem~\ref{Li}.

 \begin{thm}\label{mainthm}
Suppose $G = (G_1, \ldots, G_n)$ is a collection of $n$ graphs with the same vertex set $V$ and $\sigma_2(G_i) \geq n$ for all $i \in [n]$. Then for every pair of vertices $u,v \in V$, one of the following statements holds:
\begin{enumerate}
    \item[$(B1)$] $G$ contains a rainbow Hamiltonian $u,v$-path;
    \item[$(B2)$] 
    There exists a partition $V\setminus \{x,y\} = X \cup Y$ with $|X|,|Y| \geq 1$ such that for every $i \in [n]$, in $G_i$, $G_i[X]$ and $G_i[Y]$ are disjoint cliques, and $u$ and $v$ are adjacent to all of $X \cup Y$.
    \item[$(B3)$] $n$ is even and there exists a partition $V = X \cup Y$ with $|X| = n/2, |Y| = n/2,$ $u,v \in X$, such that for every $i \in [n]$, $G_i[X,Y]$ is complete bipartite.
\end{enumerate} 
\end{thm}

One may check that if either $(B2)$ or $(B3)$ apply for any $u,v \in V$, then $G$ has a rainbow Hamiltonian cycle. As a result we obtain the following.

\begin{cor}\label{hamorhamconn}
    Suppose $G = (G_1, \ldots, G_n)$ is a collection of $n$ graphs with the same vertex set $V$ and $\sigma_2(G_i) \geq n$ for all $i \in [n]$. Then $G$ either has a rainbow Hamiltonian cycle or is rainbow Hamiltonian-connected.
\end{cor}

In fact, we will prove a stronger version of Theorem~\ref{mainthm} in which we can also fix some rainbow linear forest to embed within our rainbow $u,v$-Hamiltonian path. First we need some definitions.

A {\bf linear forest} is a forest in which every component is a path.

\begin{definition}
Given a linear forest $H$, a pair of vertices $u,v$ is called $H$-compatible if $d_H(u) \leq 1$ and $d_H(v) \leq 1$, and $u$ and $v$ belong to different components of $H$. 
\end{definition}

Note that if $|E(H)| < n-1$, a necessary condition for a Hamiltonian $u,v$-path to contain a linear forest $H$ as a subgraph is for $H$ to be $u,v$-compatible.

Throughout this paper, for simplicity we will assume that the linear forest $H$ contains distinct components $H_u$ and $H_v$ containing $u$ and $v$ respectively as endpoints (these components may be just singleton vertices). Moreover, if $H$ contains other nontrivial components, we write these components as $H_1, H_2, \ldots, H_q$ where each $H_i$, $i \in [q]$, contains at least one edge. 

If $F$ is a rainbow subgraph of $G$, then we use $c(F) = \{c(e): e \in F\}$ to denote the set of colors used in $F$.

\begin{thm}\label{mainthm2}
Fix integers $k,n$ with $k \leq (n-4)/3$. Suppose $G = (G_1, \ldots, G_n)$ is a collection of $n$ graphs with the same vertex set $V$ and $\sigma_2(G_i) \geq n+k$ for all $i \in [n]$. Fix any rainbow linear forest $H$ in $G$ with $k$ edges. Then for any $H$-compatible pair of vertices $u,v \in V$, one of the following statements holds:
\begin{enumerate}
    \item[$(C1)$] There exists a rainbow Hamiltonian $u,v$-path which contains $H$;
    \item[$(C2)$] The linear forest $H$ has exactly two components $H_u$ and $H_v$, and there exists a partition $X\cup Y$ of $V(G) \setminus V(H)$ with $|X|,|Y| \geq 1$ such that for every unused color $i \in [n] \setminus c(H)$, in $G_i$, $X$ and $Y$ form disjoint cliques, there are no edges between $X$ and $Y$, and each vertex in $H$ is adjacent to every vertex in $X \cup Y$. 

    \item[$(C3)$] There is a partition $V = X \cup Y$ with $|X| = (n+k)/2, |Y| = (n-k)/2$, such that $H \subseteq X$, and for every $i \in [n] \setminus c(H)$, $G_i[X,Y]$ is complete bipartite and $G_i[Y]$ is independent. 
\end{enumerate} 
\end{thm}

Setting $k=0$ implies Theorem~\ref{mainthm}.

\subsection{Notation and proof outline}

Let $G = (G_1, G_2, \ldots, G_n)$ be a collection of graphs on the same vertex set $V = V(G)$. We define $N_i(x)=N_{G_i}(x)$ to to be the set of neighbors of $x$ in $G_i$, and $d_i(x)=|N_i(x)|$ be the degree of $x$ in $G_i$ for $i \in [n]$. 

As Theorem~\ref{mainthm2} is a strengthening of Theorem~\ref{mainthm}, we prove only the former. We fix a linear forest $H$ and a pair of vertices $u$, $v$ that are $H$-compatible. We remove a subset of the vertices of $H$ from each $G_i$ and show that the resulting collection of graphs satisfy the conditions of Theorem~\ref{Li}, and thus one of $(A1)$, $(A2)$, or $(A3)$ applies. We consider each of these cases separately and analyze the behavior of the graphs when we add the deleted vertices back in. For example in the case of $(A1)$, we start with a rainbow spanning path on the remaining vertices and show that the deleted portions of the linear forest $H$ can be ``absorbed" back into this path to find a rainbow Hamiltonian $u,v$-path containing $H$. 

For induction purposes, we prove an easy extension of Theorem \ref{Li}. 
\begin{lem}\label{Li2} Let $k \geq 1$ be an integer. Suppose $G = (G_1, \ldots, G_{n+k})$ is a collection of $(n+k)$ graphs where $k\geq 0$ with the same vertex set $V$ where $|V|=n$ and $\sigma_2(G_i) \geq n-2$ for all $i \in [n+k]$. Then one of the following statements holds:
\begin{enumerate}
    \item[$(A1')$] $G$ has a rainbow Hamiltonian path;
    \item[$(A2')$] $G$ consists of $n+k$ identical copies of $K_\ell \cup K_{n-\ell}$ where $\ell \in [n-1]$;
    \item[$(A3')$] $n$ is even and there is a partition $V = X \cup Y$ with $|X| = n/2-1, |Y| = n/2+1,$ and for every $i \in [n+k]$, $G_i[Y]$ is an independent set.
\end{enumerate}
\end{lem}
\begin{proof}Suppose that $G$ has no rainbow Hamiltonian path. In particular, for any index set of size $n$, $\{i_1, \ldots, i_n\} \subset [n+k]$, $(G_{i_1}, \ldots, G_{i_{n}})$ does not have such a path.

We apply Theorem~\ref{Li} to $G' = (G_1, \ldots, G_n)$. 
Suppose first that $(A2)$ holds, i.e., $G_1, \ldots, G_n$ are identical copies of $K_\ell \cup K_{n-\ell}$. For $i \in [k]$, the sequence $G''= (G_1, \ldots, G_{n-1}, G_{n+i})$ must also satisfy $(A2)$ by Theorem~\ref{Li}, thus $G_{n+i}$ is identical to $G_1$, and $(A2')$ holds for all of $G= (G_1, \ldots, G_{n+k})$. 

Similarly, if $(A3)$ holds for $G'$, then $(A3')$ must also hold for all of $G$. 
\end{proof}


\section{Proof of Theorem \ref{mainthm2}}
 Fix a rainbow linear forest $H$ (along with its colors) and a pair of vertices $u,v \in V$ that are $H$-compatible. We will show that either there exists a rainbow $u,v$-Hamiltonian path that contains $H$ or either $(C2)$ or $(C3)$ hold. Without loss of generality, assume the colors used in $H$ are $c(H)=\{n-k+1, \ldots, n\}$. 
 
We may assume the connected components of $H$ are $\{H_1, \ldots, H_q\} \cup \{H_u, H_v\}$ where $u \in V(H_u)$, $v \in V(H_v)$, and each of $H_1, \ldots, H_q$ contains at least one edge.
For each $H_i$ such that $i \in [q]$, let $v_i$ and $w_i$ be two end points of $H_i$. If $H_u$ contains at least two vertices, let $w_u$ be the other endpoint of $H_u$ that is not $u$. Otherwise, let $w_u=u$. Similar for $w_v$.

Define $D= V(H) \setminus \{v_1, \ldots, v_q\}$. It is easy to see that $u,v \in D$, $|D|=k+2$, and $|V(H)|=k+2+q$. Let $G'=(G'_1, \ldots, G'_{n-k})$ where each $G'_i$ is obtained from $G_i$ by deleting the vertex set $D$. We have $|V(G')|=n-k-2$. Observe that 
\begin{equation} \label{deg3}
\sigma_2(G'_i) \geq n+k-2(k+2)=(n-k-2)-2=|V(G_i')|-2 \ \mbox{for all} \ i \in [n-k].    
\end{equation} 

By Lemma \ref{Li2}, one of $(A1')$, $(A2')$ or $(A3')$ holds for $G'$.

{\bf Case 1}: $(A1')$ holds for $G'$. That is, $G'$ contains a Hamiltonian rainbow path $P$ of order $n-k-2$, 
using $|V(G')|-1= n-k-3$ distinct colors in $[n-k]$. Let $S$ be the set of unused colors (i.e., $S=[n] \setminus (c(P) \cup c(H))$), where $|S| = 3$. For each $i \in [q]$, since $v_i \notin D$, we have $v_i \in V(P)$ (See Figure ~\ref{fig:1}). 

\begin{figure}[ht!] 
    \centering
    \begin{tikzpicture}
        \coordinate (v1) at (-2,0);
        \coordinate (v2) at (-1,0);
        \coordinate (v3) at (0,0);
        \coordinate (v4) at (4,0);
        \coordinate (v5) at (7,0);
        \coordinate (v6) at (-1,2);
        \coordinate (v7) at (0,2);
         \coordinate (v8) at (4,2);
        \coordinate (v9) at (-1,3);
         \coordinate (v10) at (1,3); 
          \coordinate (v11) at (3,3);
         \coordinate (v12) at (5,3); 
        \draw[fill=black] (v1) circle (3pt);
        \draw[fill=black] (v2) circle (3pt);
        \draw[fill=black] (v3) circle (3pt);
        \draw[fill=black] (v4) circle (3pt);
        \draw[fill=black] (v5) circle (3pt);
         \draw[fill=black] (v6) circle (3pt);
         \draw[fill=black] (v7) circle (3pt);
        \draw[fill=black] (v8) circle (3pt);
        \draw[fill=black] (v9) circle (3pt);
        \draw[fill=black] (v10) circle (3pt);
        \draw[fill=black] (v11) circle (3pt);
        \draw[fill=black] (v12) circle (3pt);
        \draw[fill=black,thick] (v1)--(v2)--(v3)--(v4)--(v5);
        \draw[fill=black,thick] (v2)--(v6);
        \draw[fill=black,thick] (v3)--(v7);
         \draw[fill=black,thick] (v4)--(v8);
          \draw[fill=black,thick] (v9)--(v10);
           \draw[fill=black,thick] (v11)--(v12);
        \node at (7.6,0) {$P$};
         \node at (-1,2.3) {$w_1$};
         \node at (-1,-0.3) {$v_1$};
          \node at (-0.7,1) {$H_1$};
           \node at (0,2.3) {$w_2$};
         \node at (0,-0.3) {$v_2$};
          \node at (0.3,1) {$H_2$};
          \node at (4,2.3) {$w_q$};
         \node at (4,-0.3) {$v_q$};
          \node at (4.3,1) {$H_q$};
           \node at (-1,3.3) {$w_u$};
         \node at (1,3.3) {$u$};
          \node at (0,3.3) {$H_u$};
           \node at (3,3.3) {$v$};
         \node at (5,3.3) {$w_v$};
          \node at (4,3.3) {$H_v$};
          \node at (2,1) {$\dots \dots \dots \dots$};
        \coordinate (dummy) at (6,-1);
        \draw[fill=white,white] (dummy) circle (3pt);
    \end{tikzpicture}
    \caption{The rainbow path $P$ and linear forest $H$.}
    \label{fig:1}
\end{figure}
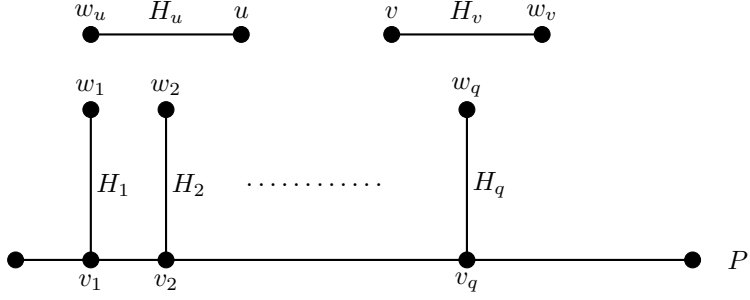

We will first ``absorb" each of $H_1, \ldots, H_q$ into the rainbow path $P$, and then attach $H_u, H_v$ as the ends of the path. We first claim that:
\begin{claim}
There exists a rainbow path $P^*$ with $V(P) \subseteq V(P^*)$ that contains all of $H_1, \ldots, H_q$.  
\end{claim}

\begin{proof}
A path is called maximal if it contains the most number of components from $H_1, \ldots, H_q$ among all paths $P'$ that satisfies the following properties:
\begin{enumerate}
    \item[(a)] $P'$ is rainbow,
    \item[(b)] $V(P) \subseteq V(P')$, and
    \item[(c)] $|[n] \setminus (c(P') \cup c(H))|=3$.
\end{enumerate}

For a maximal path $P'$, define $i(P')=|\{j \in [q]: H_j \subseteq P'\}|$. Let $P'$ be a maximal path, say with $i(P')=q'$. We can assume that $q'<q$ otherwise we can take $P'$ satisfies the claim. A maximal path always exists since $P$ satisfies (a), (b), and (c). 

Let $P'=x_1 \ldots x_{\ell'}$ and let $\mathcal C_{P'}$ be the collection of $q'$ components of $H$ that are contained in $P'$.  If $q'\geq 1$, without loss of generality assume $\mathcal C_{P'} = \{H_1, \ldots, H_{q'}\}$.
Otherwise if $q'=0$, $\mathcal C_{P'} = \emptyset$ and $\ell'=n-k-2$.

Without loss of generality, assume $c(x_jx_{j+1})=j$ for all $j \in [\ell'-1]$. Let $S_{P'}$ be the new set of unused colors for $P'$ and $H$ (i.e., $S_{P'}=[n] \setminus (c(P') \cup c(H))$). By condition (c), $|S_{P'}|=3$. Since $V(P) \subseteq V(P')$, we have $v_i \in V(P')$ for all $i \in [q]$. Say $v_i=x_{j_i}$. Suppose the components of $\mathcal C_{P'}$ together contain $g'$ edges and thus $g'+q'$ vertices. Since each $H_i$ contains a vertex $v_i \in V(P)$, we have \[V(P')=V(P) \cup \bigcup_{H_i \in \mathcal C_{P'}} (V(H_i) \setminus \{v_i\}).\]
Thus, $\ell'=|P'|=|P|+g'=n-k-2+g'$. Define $D_{P'}=V(G) \setminus V(P')$. We have $|D_{P'}|=k+2-g'$. Now consider $H_t \notin \mathcal C_{P'}$. Recall $w_t$ is the endpoint of $H_t$ that is not on $P'$, and $v_t=x_{j_t}$ is the other endpoint of $H_t$ in $P'$.

If $x_1w_t \in E(G_a)$ for some unused color $a \in S_{P'}$, then \[P''=x_{j_t-1} \ldots x_1 w_t H_t v_t x_{j_t+1}\ldots x_{\ell'}\]
with $c(x_1w_t)=a$ is a rainbow path with $\ell'+|H_t|-1$ vertices containing all of $\mathcal C_{P'} \cup \{H_t\}$ . We observe that $S_{P''}=(S_{P'} \cup \{j_t-1\}) \setminus \{a\}$ since the edge $x_{j_t-1}x_{j_t}$ is not used in $P''$. Thus $|S_{P''}|=3$. This contradicts the choice of $P'$ as a maximum path. Therefore, $x_1w_t \notin E(G_a)$ for all $a \in S_{P'}$. By a symmetric argument, we also have $x_{\ell'}w_t \notin E(G_a)$ for all $a \in S_{P'}$.

Since $x_1w_t \notin E(G_a)$ for all $a \in S_{P'}$, then for any two colors say $a_1,a_2 \in S_{P'}$, we have 
\begin{align*}
d_{a_1}(x_1)+d_{a_1}(w_t) &\geq \sigma_2(G_{a_1}) \geq  n+k,\\ 
d_{a_2}(x_1)+d_{a_2}(w_t) &\geq \sigma_2(G_{a_2}) \geq n+k.
\end{align*}
This implies either $d_{a_1}(x_1)+d_{a_2}(w_t) \geq n+k$ or $d_{a_2}(x_1)+d_{a_1}(w_t) \geq n+k$. Without loss of generality, assume $d_{a_2}(x_1)+d_{a_1}(w_t) \geq n+k$. Define 
\[N_{P',a_1}(w_t)=\{x_i \in N_{a_1}(w_t) \cap V(P')\},\]
\[N_{P',a_2}(x_1)=\{x_i \in N_{a_2}(x_1) \cap V(P')\},\]
\[N^-_{P',a_2}(x_1)=\{x_i \in V(P'): x_{i+1} \in N_{P',a_2}(x_1)\}.\]
Since $w_t \in D_{P'}$, we have \[|N_{P',a_1}(w_t)| \geq d_{a_1}(w_t)-(|D_{P'}|-1)=d_{a_1}(w_t)-k-1+g'.\] Since $x_1w_t \notin E(G_{a_2})$, we have \[|N^-_{P',a_2}(x_1)|=|N_{P',a_2}(x_1)| \geq  d_{a_2}(x_1)-(|D_{P'}|-1)=d_{a_2}(x_1)-k-1+g'.\] 
Observe that $N^-_{P',a_2}(x_1) \subseteq \{x_1, \ldots, x_{\ell'-1}\}$. Since $x_1w_t, x_{\ell'}w_t \notin E(G_{a_1})$, we have $N_{P',a_1}(w_t) \subseteq \{x_2, \ldots, x_{\ell'-1}\}$. Thus, we have $N^-_{P',a_2}(x_1) \cup N_{P',a_1}(w_t) \subseteq \{x_1, \ldots, x_{\ell'-1}\}$. Then
\begin{align*}
 |N_{P',a_1}(w_t)|+|N^-_{P',a_2}(x_1)| &\geq d_{a_1}(x_1)+d_{a_2}(w_t)-2k-2+2g'\\
 &\geq n-k-2+2g'\\
 &=\ell'+g'\\
 &\geq |N_{P',a_1}(w_t) \cup N^-_{P',a_2}(x_1)|+1+g'.
\end{align*}
Thus, $|N_{P',a_1}(w_t) \cap N^-_{P',a_2}(x_1)| \geq 1+g'$. Observe that $N_{P',a_1}(w_t) \cap N^-_{P',a_2}(x_1) \subseteq \{x_2, \ldots, x_{\ell'-1}\}$.

Define the set \[T_{P'}=\{x_s \in V(P'): x_sx_{s+1} \in \{E(H_i): H_i \in \mathcal C_{P'}\}\}.\] By definition $|T_{P'}|=g'$. Then $|(N_{P',a_1}(w_t) \cap N^-_{P',a_2}(x_1)) \setminus T_{P'}|\geq 1$. Hence, there exists an $i$ such that $x_i \in N_{P',a_1}(w_t)$, $x_{i+1} \in N_{P',a_2}(x_1)$, and $x_ix_{i+1} \notin E(H_1) \cup \dotsm \cup E(H_{q'})$. We consider the following cases:

If $i \leq j_t-2$, define the rainbow path $P''=x_{j_t-1} \ldots x_{i+1} x_1 \ldots x_i w_t H_t v_t x_{j_t+1} \ldots x_{\ell'}$
with $c(w_tx_i)=a_1$ and $c(x_1x_{i+1})=a_2$. (See Figure~\ref{fig:2}.)
If $i=j_t-1$, let $P''=x_1 \ldots x_i w_t H_t x_{t+1} \ldots x_{\ell'}$
with $c(w_tx_i)=a_1$.
If $i=j_t$, let $P''=w_t H_t x_i \ldots x_1 x_{i+1} \ldots x_{\ell'}$
 with $c(x_1x_{i+1})=a_2$.
Finally, if $i \geq j_t+1$, let $P''=x_{j_t+1} \ldots x_i w_t H_t v_t \ldots x_1 x_{i+1} \ldots x_{\ell'}$
 with $c(w_tx_i)=a_1$ and $c(x_1x_{i+1})=a_2$.

In all four cases, $P''$ is a rainbow path that contains the components $\mathcal C_{P'} \cup \{H_t\}$, and $|S_{P''}| = 3$, contradicting the choice of $P'$.
\end{proof}

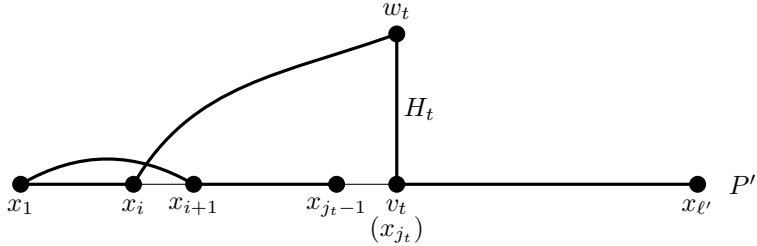
\begin{figure}[ht!]
    \centering
    \begin{tikzpicture}
        \coordinate (v1) at (-2,0);
        \coordinate (v2) at (7,0);
       \coordinate (v3) at (3,0);
       \coordinate (v4) at (3,2);
       \coordinate (v5) at (2.2,0);
       \coordinate (v6) at (-0.5,0);
       \coordinate (v7) at (0.3,0);

        \draw[fill=black] (v1) circle (3pt);
        \draw[fill=black] (v2) circle (3pt);
        \draw[fill=black] (v3) circle (3pt);
        \draw[fill=black] (v4) circle (3pt);
         \draw[fill=black] (v5) circle (3pt);
          \draw[fill=black] (v6) circle (3pt);
           \draw[fill=black] (v7) circle (3pt);
        \draw[fill=black, thin] (v1)--(v2);
        \draw[fill=black, thin] (v3)--(v4);
        \draw[very thick] (v1) to[out=30,in=150] (v7);
        \draw[very thick] (v6) to[out=60,in=-160] (v4);
        \draw[very thick] (v1)--(v6);
        \draw[very thick] (v7)--(v5);
        \draw[very thick] (v3)--(v4);
        \draw[very thick] (v3)--(v2);
        \node at (7.6,0) {$P'$};
         \node at (-2,-0.3) {$x_1$};
         \node at (7,-0.3) {$x_{\ell'}$};
          \node at (3,-0.3) {$v_t$};
           \node at (3,-0.6) {($x_{j_t}$)};
          \node at (3,2.3) {$w_t$};
          \node at (3.3,1) {$H_t$};
          \node at (2.2,-0.3) {$x_{j_t-1}$};
          \node at (-0.5,-0.3) {$x_i$};
           \node at (0.3,-0.3) {$x_{i+1}$};
        \coordinate (dummy) at (6,-1);
        \draw[fill=white,white] (dummy) circle (3pt);
    \end{tikzpicture}
    \caption{An example of the path $P''$.}
    \label{fig:2}
\end{figure}    

Let $P^*=y_1\ldots y_{\ell_1}$ be a rainbow path with $V(P)\subseteq V(P^*)$ given by the previous Claim that contains all of $H_1, \ldots, H_q$ and $|[n] \setminus (c(H) \cup c(P^*)|=3$. Without loss of generality, assume $c(y_iy_{i+1})=i$ for $i \in [\ell_1-1]$. Suppose $H_1, \ldots, H_q$ contains $g_1$ edges and thus $g_1+q$ vertices. Since for each $i \in [q]$, $H_i$ contains a vertex $v_i \in V(P)$, we have $\ell_1=|P^*|=|P|+g_1=n-k-2+g_1$. Define $D_{P^*}=V(G) \setminus V(P^*)$. We have $|D_{P^*}|=k+2-g_1$. Let $S_{P^*}$ be the set of unused colors of $H \cup P^*$. Now we attach $H_u$ onto $P^*$ using a similar proof as the previous Claim.

\begin{claim}
There exists a rainbow path $P_0$ containing $V(P^*)$ and the component $H_u$ such that $u$ is an endpoint of $P_0$ and $S_{P_0} = |[n] \setminus (c(P_0) \cup c(H))| = 2$. 
\end{claim}

\begin{proof}
Suppose such path does not exist. If $y_1w_u \in E(G_a)$ for some $a \in S_{P^*}$, then \[P_0=u H_u w_u y_1 \ldots y_{\ell_1}\] satisfies the claim. Thus, $y_1w_u \notin E(G_a)$ for all $a \in S_{P^*}$. By a symmetric argument, we also have $y_{\ell_1}w_u \notin E(G_a)$ for all $a \in S_{P^*}$. Then for any two colors say $a_1,a_2 \in S_{P^*}$ we have 
\begin{align*}
d_{a_1}(y_1)+d_{a_1}(w_u) &\geq \sigma_2(G_{a_1}) \geq  n+k,\\ 
d_{a_2}(y_1)+d_{a_2}(w_u) &\geq \sigma_2(G_{a_2}) \geq n+k.
\end{align*}
This implies either $d_{a_1}(y_1)+d_{a_2}(w_u) \geq n+k$ or $d_{a_2}(y_1)+d_{a_1}(w_u) \geq n+k$. Without loss of generality, assume $d_{a_2}(y_1)+d_{a_1}(w_u) \geq n+k$. Define 
\[N_{P^*,a_1}(w_u)=\{y_j \in N_{a_1}(w_u) \cap V(P^*)\},\]
\[N_{P^*,a_2}(y_1)=\{y_j \in N_{a_2}(y_1) \cap V(P^*)\},\]
\[N^-_{P^*,a_2}(y_1)=\{y_j \in V(P^*): y_{j+1} \in N_{P^*,a_2}(y_1)\}.\]
Since $w_u \in D_{P^*}$, we have \[|N_{P^*,a_1}(w_u)| \geq d_{a_1}(w_u)-(|D_{P^*}|-1)=d_{a_1}(w_u)-k-1+g_1,\] and since $y_1w_u \notin E(G_{a_2})$, we have \[|N^-_{P^*,a_2}(y_1)|=|N_{P^*,a_2}(y_1)| \geq  d_{a_2}(y_1)-(|D_{P^*}|-1)=d_{a_2}(y_1)-k-1+g_1.\] 

We have $N^-_{P^*,a_2}(y_1) \cup N_{P^*,a_1}(w_u) \subseteq \{y_1, \ldots, y_{\ell_1-1}\}$, and so
\begin{align*}
 |N_{P^*,a_1}(w_u)|+|N^-_{P^*,a_2}(y_1)| &\geq d_{a_1}(y_1)+d_{a_2}(w_u)-2k-2+2g_1\\
 &\geq n-k-2+2g_1\\
 &= \ell_1+g_1\\
 &\geq |N_{P^*,a_1}(w_u) \cup N^-_{P^*,a_2}(y_1)|+1+g_1.
\end{align*}
Thus, $|N_{P^*,a_1}(w_u) \cap N^-_{P^*,a_2}(y_1)| \geq 1+g_1$. Observe that $N_{P^*,a_1}(w_u) \cap N^-_{P^*,a_2}(y_1) \subseteq \{y_2, \ldots, y_{\ell_1-1}\}$. Define the set \[T_{P^*}=\{y_s \in V(P^*): y_sy_{s+1} \in E(H_1) \cup \dotsm \cup E(H_{q})\}.\] By definition, $|T_{P^*}|=g_1$. Then $|N_{P^*,a_1}(w_u) \cap N^-_{P^*,a_2}(y_1) \setminus T_{P^*}|\geq 1$. Hence, there exists an $i$ such that $y_i \in  N_{P^*,a_1}(w_u) $, $y_{i+1} \in N_{P^*,a_2}(y_1)$, and $y_iy_{i+1} \notin E(H_1) \cup \dotsm \cup E(H_{q})$. Then we have \[P_0=u H_u w_u y_i \ldots y_1 y_{i+1} \ldots y_{\ell_1}\] with $c(w_uy_i) = a_1, c(y_1y_{i+1}) = a_2$ is a rainbow path with $\ell_1+|H_u|$ vertices containing all of $H_1, \ldots, H_q, H_u$, and $u$ is an end point of $P_0$. Moreover, we observe that $S_{P_0}=(S_{P^*} \cup \{i\}) \setminus \{a_1, a_2\}$ so $|S_{P_0}|=2$.
\end{proof}

Let $P_0=z_1\ldots z_{\ell_2}$ be the rainbow path obtained from the previous Claim, where $u=z_{\ell_2}$. Without loss of generality, assume $c(z_iz_{i+1})=i$ for $i \in [\ell_2-1]$. Suppose $H_1, \ldots, H_q, H_u$ contains $g_2$ edges and thus $g_2+q+1$ vertices. Since each $H_i$ contains a vertex $v_i \in V(P)$ for $i \in [q]$ and $H_u$ contains no vertex on $P$, we have $\ell_2=|P_0|=n-k-2+g_2+1=n-k-1+g_2$. Define $D_{P_0}=V(G) \setminus V(P_0)$. We have $|D_{P_0}|=k+1-g_2$. Now we attach $H_v$ to the other end of $P_0$. The resulting path will be a rainbow Hamiltonian $u,v$-path containing $H$, thus showing that $(C1)$ in Theorem~\ref{mainthm2} holds. The remaining proof is very similar to that of the previous Claim so we omit some details.

If $z_1w_v \in E(G_a)$ for some $a \in S_{P_0}$, then \[P_{1}=v H_v w_v z_1 \ldots z_{\ell_2}\] with $c(z_1w_v)=a$ satisfies $(C1)$. So assume $z_1w_v \notin E(G_a)$ for all $a \in S_{P_0}$. Then for the two colors say $a_1,a_2 \in S_{P_0}$ we have 
\begin{align*}
d_{a_1}(z_1)+d_{a_1}(w_v) &\geq \sigma_2(G_{a_1}) \geq  n+k,\\ 
d_{a_2}(z_1)+d_{a_2}(w_v) &\geq \sigma_2(G_{a_2}) \geq n+k.
\end{align*}
Assume $d_{a_2}(z_1)+d_{a_1}(w_v) \geq n+k$, and define 
\[N_{P_0,a_1}(w_v)=\{z_j \in N_{a_1}(w_v) \cap V(P_0)\},\]
\[N_{P_0,a_2}(z_1)=\{z_j \in N_{a_2}(z_1) \cap V(P_0)\},\]
\[N^-_{P_0,a_2}(z_1)=\{z_j \in V(P_0): z_{j+1} \in N_{P_0,a_2}(z_1)\}.\]
Since $w_v \in D_{P_0}$, we have \[|N_{P_0,a_1}(w_v)| \geq d_{a_1}(w_v)-(|D_{P_0}|-1)=d_{a_1}(w_u)-k+g_2.\] 
Since $z_1w_v \notin E(G_{a_2})$, we have \[|N^-_{P_0,a_2}(z_1)|=|N_{P_0,a_2}(z_1)| \geq  d_{a_2}(z_1)-(|D_{P_0}|-1)=d_{a_2}(y_1)-k+g_2.\] 

Since $N^-_{P_0,a_2}(z_1) \subseteq \{z_1, \ldots, z_{\ell_2-1}\}$ and $N_{P_0,a_1}(w_v) \subseteq \{z_2, \ldots, z_{\ell_2}\}$, we have $N^-_{P_0,a_2}(z_1) \cup N_{P_0,a_1}(w_v) \subseteq \{z_1, \ldots, z_{\ell_2}\}$, and so 
\begin{align*}
 |N_{P_0,a_1}(w_v)|+|N^-_{P_0,a_2}(z_1)| &\geq d_{a_1}(z_1)+d_{a_2}(w_v)-2k+2g_2\\
 &\geq n-k+2g_2\\
&= \ell_2+1+g_2\\
 &\geq |N^-_{P_0,a_2}(z_1)\cup N_{P_0,a_1}(w_v)|+1+g_2.
\end{align*}
Thus, $|N^-_{P_0,a_2}(z_1)\cap N_{P_0,a_1}(w_v)| \geq 1+g_2$. Observe that $N^-_{P_0,a_2}(z_1)\cap N_{P_0,a_1}(w_v) \subseteq \{z_2, \ldots, z_{\ell_2-1}\}$. Define the set \[T_{P_0}=\{z_s \in V(P_0): z_sz_{s+1} \in E(H_1) \cup \dotsm \cup E(H_{q}) \cup E(H_u)\}.\] By definition, $|T_{P_0}|=g_2$. Then $|N^-_{P_0,a_2}(z_1)\cup N_{P_0,a_1}(w_v) \setminus T_{P_0}|\geq 1$ and there exists an $i$ such that $z_i \in  N_{P_0,a_1}(w_v) $, $z_{i+1} \in N_{P_0,a_2}(z_1)$, and $z_iz_{i+1} \notin E(H_1) \cup \dotsm \cup E(H_{q}) \cup E(H_u)$. The rainbow path \[P_{1}=v H_v w_v z_i \ldots z_1 z_{i+1} \ldots z_{\ell_2}\] with $c(w_vz_i) = a_1, c(z_1z_{i+1}) = a_2$ satisfies $(C1)$ of Theorem~\ref{mainthm2}.

\bigskip

{\bf Case 2}: $(A2')$ holds for $G'$. That is, 
$G'$ contains $n-k$ identical copies of $K_\ell \cup K_{n-k-2-\ell}$ where $\ell \in [n-k-3]$. Set $X=V(K_\ell)=\{x_1, \ldots, x_{\ell}\}$ and $Y=V(K_{n-k-2-\ell})=\{x_{\ell+1}, \ldots, x_{n-k-2}\}$.

For any $i \in [n-k]$, consider a pair of non-adjacent vertices $x$ and $y$ in $G'_i$ in which $x \in X$ and $y \in Y$. We have \[d_{G'_i}(x) + d_{G'_i}(y)=(\ell-1)+(n-k-2-\ell-1)=n-k-4\] for all $i \in [n-k]$. Since $\sigma_2(G_i) \geq n+k$ and $|D|=k+2$, each vertex in $D$ must be adjacent to both $x$ and $y$. Moreover, $i$, $x$, and $y$ were chosen arbitrarily so $D$ is adjacent to all vertices in each $G'_i$ for all $i \in [n-k]$. Then for all $i \in [n-k]$, $G_i$ contains the subgraph $(K_\ell \cup K_{n-k-2-\ell})\vee D$  (on vertex sets $X,Y,$ and $D$). If $H = H_u \cup H_v$, then $G$ satisfies $(C2)$. Otherwise recall  each component $H_i$ in $H - H_u - H_v$ has a distinct vertex $v_i \in V(G')$. Say $v_i = x_{j_i}$, and let $W = \{ x_{j_1}, \ldots, x_{j_q}\}$. 

Since $|V(G')| = n-k-2 \geq k+1 \geq q+1$, either $X$ or $Y$ contains a vertex outside of $W$, say it is $Y$. Let $P_1$ be any spanning rainbow path of $X$ with colors in $[n-k]$ that ends with a vertex in $W$, say $x_{i_1}$, (recall that $X$ and $Y$ are cliques in all unused colors), and let $P_2$ be any spanning rainbow path of $Y$ with colors in $[n-k] - c(P_1)$ that does not end in a vertex of $W$. 

We obtain a new rainbow path $P_3$ by attaching $H_1$ to the end of $P_1$ (at $x_{j_1}$) and then adding an edge in any unused color between the other end of $H_1$ and the beginning of $P_2$. This is possible because vertices in $D$ are adjacent to all vertices of $V(G')$ in $G_i$, $i \in[n-k]$. We similarly insert the components $H_2, \ldots, H_q$ into $P_3$ by including each $H_i$ after $x_{j_i}$. We obtain a new rainbow path $P_4$ containing $H_1 \cup \ldots \cup H_q$. Then we can attach $H_u$ and $H_v$ to the beginning and end of $P_4$ respectively to obtain a path satisfying $(C2)$ (see Figure~\ref{fig:a2'}).

\begin{figure}
  \centering \includegraphics[width=0.6\linewidth]{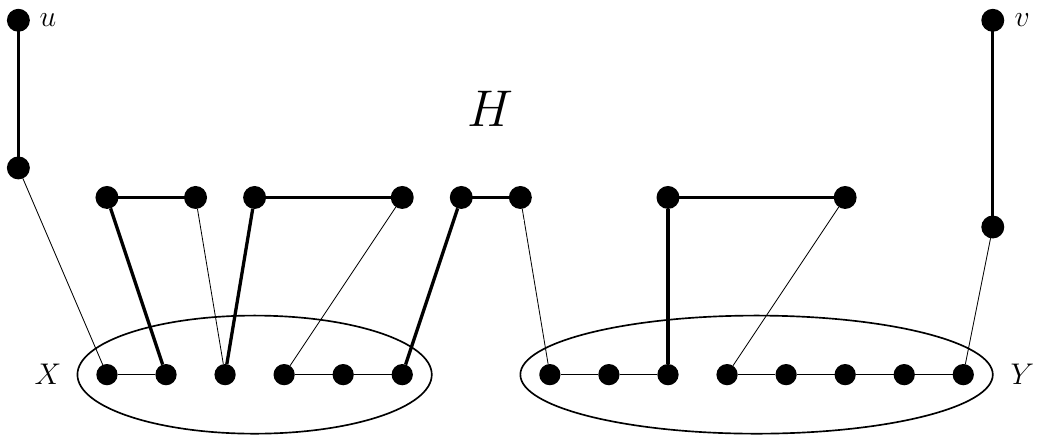}
  \caption{A path in Case 2.}
  \label{fig:a2'}
\end{figure}


{\bf Case 3}: $(A3')$ holds for $G'$. We have $n-k-2$ is even and there is a partition $V(G')=X' \cup Y'$ with $|X'|=(n-k-2)/2-1 = (n-k)/2 - 2$, $|Y'|=(n-k-2)/2+1 = (n-k)/2$, and for every $i \in [n-k]$, $G_i[Y']$ is an independent set. For any $i \in [n-k]$, consider a pair of non-adjacent vertices $x$ and $y$ in $G'_i$ in which $x, y \in Y'$. We have $\sigma_2(G'_i) \leq d_{G'_i}(x) + d_{G'_i}(y) \leq 2((n-k-2)/2-1)=n-k-4$ for all $i \in [n-k]$. Since $\sigma_2(G_i) \geq n+k$ and $|D|=k+2$, each vertex in $D$ must be adjacent to both $x$ and $y$. Moreover, $i$, $x$, and $y$ were chosen arbitrarily so $D$ is adjacent to all vertices in $Y'$ in each $G'_i$ for all $i \in [n-k]$. Set $X=X' \cup D$ and $Y=Y'$. We have $|X|=|X'|+|D|=(n+k)/2$ and $|Y|=|Y'|=(n-k)/2$. It is easy to see $G_i[X,Y]$ is complete bipartite for every $i \in [n-k]$. 

If $H \subseteq X$ then we satisfy (C3) in Theorem~\ref{mainthm2}. So suppose $H$ intersects $Y$. In particular, at least one vertex in $H-D = \{v_1, \ldots, v_q\}$ belongs to $Y$. 

Observe that
\begin{equation}\label{Gidegree}
\mbox{$\sigma_2(G_i[X]) \geq 2k$ for $i \in [n-k]$.}
\end{equation}
Indeed, for any two non-adjacent vertices $x,y \in X$, we have 
\begin{align*}
d_{G_i[X]}(x)+d_{G_i[X]}(y)&=d_{G_i}(x)+d_{G_i}(y)-d_{G_i \cap Y}(x)-d_{G_i \cap Y}(y)\\
&\geq n+k-(n-k)/2-(n-k)/2=2k. 
\end{align*}

In order to prove that there exists a rainbow $u,v$-Hamiltonian path containing $H$, we show that in fact we can find such a path covering a larger linear forest $H'$ which itself contains $H$. First we construct $H'$.

\begin{claim}
    There exists a rainbow linear forest $H'$ such that
    \begin{enumerate}
        \item[(a)] $u,v$ are $H'$-compatible,
        \item[(b)] $H \subseteq H'$, 
        \item[(c)] every edge in $E(H') \setminus E(H)$ contains a vertex in $X'$ the other vertex in $X' \cup \{w_1, \ldots, w_q, w_u, w_v\}$, and
        \item[(d)] exactly $q-1$ edges in $H'$ intersect $X'$. (Some of these edges may be in $H$.)
    \end{enumerate}
\end{claim}

\begin{proof}
Recall that $v_1, \ldots, v_q$ are the end vertices of the paths $H_1, \ldots, H_q$ respectively that were not included in the set $D$ (and so $v_i \in V(G')$). For each $i$ such that $v_i \in X'$, the component $H_i$ contains exactly one edge incident to $X'$. Let the (possibly empty) union of these edges be the matching $M$.

We extend $M$ to a rainbow linear forest by adding to it edges contained within $X'$ with colors in $[n-k]$. Let $F$ be a maximal such rainbow linear forest. If $|E(F)| \geq q-1$, then $H' = F \cup H$ satisfies the claim. So suppose $|E(F)| = t \leq q-2$. 

We may consider $F$ to be a forest covering $X'$, where some vertices may be isolated in $F$. Call a vertex $x$ {\em usable} if $x \in X'$ and $d_{F}(x) \leq 1$. Call a vertex $x \in X'$ {\em robust in color $i$} if $x$ is usable and $x$ has at least $q-t$ neighbors in $\{w_1, \ldots, w_q,w_u, w_v\}$ in $G_i$. Let $S_{H \cup F}=[n] \setminus c(H \cup F)$ be the set of unused colors for $H \cup F$.
We will show the following.

\begin{equation}\label{robust}
    \mbox{There exists $q-t-1$ vertices that are each robust in a distinct color from $S_{H \cup F}$.}
\end{equation}

Suppose first that $z \in X'$ is a usable vertex satisfying $d_{G_i[X]}(z) \geq k$ for some color $i \in S_{H \cup F}$. We will show that $z$ is robust in color $i$.  If $z$ has a usable neighbor $z'$ in $G_i$ in a different component of $F$, then we can extend $F$ to a bigger rainbow linear forest by adding $zz'$ in color $i$, contradicting the maximality of $F$. Thus all $X'$-neighbors of $z$ in $G_i$ are either unusable or in the same component of $F$ as $z$.

Suppose $P$ is a component of $F$ and $|E(P)| = r\geq 1$. If $P$ contains no edges from the matching $M$, then it has $r+1$ vertices in $X'$, of which $r+1-2=r-1$ are unusable. If $P$ contains exactly one edge from $M$, then it has $r$ vertices in $X'$, and at least one endpoint of $P$ is usable. Thus there are also at most $r-1$ unusable $X'$ vertices in $P$. Similarly, if $P$ contains two edges of $M$, then it has $r+1-2 = r-1$ vertices in $X'$, each of which may be unusable. Summing up over all components of $F$, we have at most $t-1 + 1$ vertices that are either unusable or in the same component of $F$ as $z$ (namely the other endpoint of the component). Thus $z$ has at least $d_{G_i[X]}(z) - t \geq k-t$ neighbors in $G_i$ outside of $X'\cup Y'$.

Since $H - (X' \cup Y')$ contains $k+2$ vertices including the $q+2$ vertices in $\{w_1, \ldots, w_q, w_u, w_v\}$, the number of vertices in $\{w_1, \ldots, w_q, w_u, w_v\}$ that are adjacent to $z$ in $G_i$ is at least
\[(k-t) - [(k+2) - (q+2)] = q - t .\] 

To prove~\eqref{robust} it remains to prove that there is a set of distinct vertices $\{z_1, \ldots, z_{q-1-t}\}$ and a set of distinct unused colors of $S_{H \cup F}$, say $\{1, \ldots, q-1-t\}$ such that $d_{G_i[X]}(z_i) \geq k$. As before, if a component of $F$ contains at most one edge from $M$, then it has at least one usable endpoint. Let $m$ denote the number of components of $F$ with two edges in $M$. Then $F[X']$ contains at most $t - 2m$ edges, and hence at least $|X'| - (t-2m)$ components (some of them may be singletons), of which at least $|X'| - (t-2m) -m \geq (n-k)/2 -2 -t \geq q - t$ contain usable vertices (since $(n-k)/2 - 2 \geq k \geq q$).

Let $Z=\{z_1, \ldots, z_{q-t}\}$ be a set of usable vertices from distinct components of $F$ and suppose the colors $\{1, \ldots, q-t-1\}$ belong to $S_{H \cup F}$. By the maximality of $F$, $Z$ is an independent set in every $G_i$ for $i \in [q-t-1]$.  Choose any pair $\{z_a, z_b\} \subset Z$ and any $i \in [q-t-1]$. Equation~\eqref{Gidegree} implies either $d_{G_i[X]}(z_a) \geq k$ or $d_{G_i[X]}(z_b) \geq k$. Thus one of these vertices is robust in color $i$. We may remove this robust vertex from $Z$ and remove the color $i$. This process can be performed $q-t-1$ times to find $q-t-1$ robust vertices, each corresponding to a distinct color. This proves~\eqref{robust}. 

If there exists a rainbow matching $M_0$ between $Z$ and $\{w_1, \ldots, w_q,w_u,w_v\}$ such that $|M_0|=q-t-1$ and $H \cup F \cup M_0$ is a rainbow linear forest, then we are done by taking $H'=H \cup F \cup M_0$. Let $M_0$ be a maximum rainbow matching between $Z$ and $\{w_1, \ldots, w_q,w_u,w_v\}$ such that 
\begin{enumerate}
    \item[(1)] The linear forest $H \cup F \cup M_0$ is rainbow.
    \item[(2)] For each $z \in Z \cap M_0$, the edge $e \in E(M_0)$ adjacent to $z$ is in the color that $z$ is robust in.
\end{enumerate}

 Then we have $|M_0|=t_1 \leq q-t-2$. Thus there exists a vertex say $z_i \in Z$ such that $z_i \notin M_0$. Assume $z_i$ is robust in color $j$. We observe that $z_i$ is contained in a common component of $H \cup F \cup M_0$ with at most one vertex in $\{w_1, \ldots, w_q, w_u, w_v\}$. Call this vertex $f(z_i)$, if it exists. Since $z_i$ has $q-t$ neighbors in $\{w_1, \ldots, w_q,w_u,w_v\}$ in $G_j$, it has $q-t-1-t_1\geq 1$ neighbors in $\{w_1, \ldots, w_q,w_u,w_v\}-f(z_i)-V(M_0)$ in $G_j$. Pick a vertex say $w_s \in \{w_1, \ldots, w_q,w_u,w_v\}-f(z_i)-V(M_0)$. Set $M'_0=M_0 \cup \{z_i w_s\}$ where $c(z_i w_s)=j$. We have $|M'_0|>|M_0|$ but $M'_0$ satisfies conditions (1) and (2) above. A contradiction to $M_0$ being a maximum matching with $|M_0| \leq q-t-2$. Hence, we must have $|M_0|=q-t-1$. The resulting rainbow linear forest $H'=H \cup F \cup M_0$ has $t+ (q-1-t) = q-1$ edges incident to $X'$ and satisfies (a) - (d).
\end{proof}

For each $i \in [n-k]$, let $G^0_i$ be the graph obtained from $G_i$ by merging each connected components of $H' \cap X$ into one vertex. Let $G^0=(G^0_1, \ldots, G^0_{n-k})$ and $V(G^0)=X^0 \cup Y^0$ where $X^0$ is the resulting $X$ after merging and $Y^0=Y$. We observe that $|X^0|=|X'|+q+2-(q-1)=(n-k)/2+1$ since exactly $q-1$ edges in $H'$ intersect $X'$. Let $C$ be a connected component of $H' \cap X$ and suppose $V(C)$ is merged into a vertex $s$, then $s$ is adjacent to a vertex $y \in Y^0$ in $G^0_i$ if all vertices of $C$ are adjacent to $y$ in $G_i$. Since $G_i[X,Y]$ is complete bipartite, $G^0_i[X^0, Y^0]$ is complete bipartite all $i \in [n-k]$.

If $H'$ contains an edge whose has an end point in $Y$, then that edge must belong to $H$. Since each connected component of $H$ has at most one vertex that is not in $D$, it has at most one edge that has an end point in $Y$ so the union of such edges forms a matching $M^0$ in $G_i[X^0, Y^0]$. It is easy to see for any two vertices in $u^0,v^0 \in X^0$, there exists a rainbow Hamiltonian $u^0,v^0$-path which contains the matching $M^0$, and uses colors that were not used by $H'$ (except on $M^0$). Uncontracting the the vertices of $H'$, we obtain a rainbow Hamiltonian $u,v$-path that contains $H$. \qed

\section{Concluding remarks}

\begin{enumerate}
\item In Theorem~\ref{mainthm2}, the bound $k \leq (n-4)/3$ on the number of edges in the linear forest $H$ is likely not best possible. We make no attempt to optimize it in this paper. Perhaps a better option would be a bound in terms of the specific linear forest $H$. For instance if $H$ is a matching on $k$ edges, then clearly $k \leq n/2$ is necessary. However it should be possible to embed other types of linear forests on more than $n/2$ edges.   
    \item In graphs, Hamiltonian-connectedness implies Hamiltonicity. This same implication is not apparent for rainbow Hamiltonian-connectedness. Suppose $G = (G_1, \ldots, G_n)$ is a collection of $n$ graphs with the same vertex set $V$ where $|V|=n$ and $\sigma_2(G_i) \geq n$ for all $i \in [n]$.
By Corollary~\ref{hamorhamconn}, if $G$ does not contain a rainbow Hamiltonian cycle, then for every $u,v \in V(G)$, there is a $u,v$-Hamiltonian path using $n-1$ colors. If the edge $uv$ is available in the remaining color, we could complete the path to a rainbow Hamiltonian cycle. 
In particular, this holds if there exists an edge $uv$ that belongs to all $G_i$, $i \in [n]$.

We say a vertex $u$ is {\em small in color $i$} if $d_i(u) < n/2$. Since $\sigma_2(G_i) \geq n$, the small vertices in $G_i$ form a clique. Thus if there exists a pair of vertices that are small in every $G_i$, then $G$ contains a rainbow Hamiltonian cycle.
Does the same hold true if there is only one vertex that is small in every $G_i$?

\item We consider a weaker version of Question~\ref{rainbowham}:

\begin{quest}Does there exist a global constant $k$ such that the following holds: for sufficiently large $n$, if $G = (G_1, \ldots, G_n)$ is a collection of $n$ graphs with the same vertex set $V$ where $|V|=n$ and $\sigma_2(G_i) \geq n+k$ for all $i \in [n]$, then $G$ contains a rainbow Hamiltonian cycle?
\end{quest}
With such a condition on $\sigma_2$, one may apply Theorem~\ref{mainthm2} and predetermine a linear forest (including a choice of colors for the forest) to be used in some rainbow $u,v$-Hamiltonian path. We cannot control the unused color of the path, however we may use $H$ to limit its options.
\end{enumerate}

{\bf Acknowledgement.} The second author thanks Debsoumya Chakraborti for helpful discussions on this topic.


\begin{thebibliography}{99}
\small


\bibitem{B}
P. Bradshaw: 
Transversals and Bipancyclicity in Bipartite Graph Families,
{\em The Electronic Journal of Combinatorics}, \textbf{28} (2021).

\bibitem{CSWW}
Y. Cheng, W. Sun, G. Wang,  L. Wei: 
Transversal Hamilton paths and cycles,
arXiv:2406.13998 (2024).

\bibitem{JK}
F. Joos, J. Kim: 
On a rainbow version of Dirac’s theorem,
{\em Bulletin of the London Mathematical Society}, \textbf{52} (2020), no. 3, 498--504.

\bibitem{LLL}
L. Li, P. Li, X. Li: 
Rainbow structures in a collection of graphs with degree conditions,
{\em Journal of Graph Theory}, \textbf{104} (2023), no. 2, 341–359.


\bibitem{Ore}
O. Ore: 
Note on Hamilton Circuits,
{\em The American Mathematical Monthly}, \textbf{67} (1960), no. 1, 55.


\bibitem{SM}
E. Schmeichel, J. Mitchem: 
Bipartite graphs with cycles of all even lengths,
{\em Journal of Graph Theory}, \textbf{6} (1982), no. 4, 429–439.


\bibitem{SWW2}
W. Sun, G. Wang, L. Wei: 
Transversal Structures in Graph Systems: A Survey,
arXiv:2412.01121  (2024).


\bibitem{SWW}
W. Sun, G. Wang, L. Wei: 
Transversal Panconnectedness in Graph Collections,
{\em The Electronic Journal of Combinatorics}, \textbf{32} (2025).





\end{thebibliography}
\end{document}